\documentclass[a4paper,12pt]{amsart}

\pdfoutput=1
\usepackage{setspace}
\usepackage[left=25mm,head=30mm,bottom=30mm,foot=20mm,right=25mm]{geometry}
\usepackage{amsmath,amsthm,amsfonts,amssymb,mathtools,mathrsfs,url,bm,enumitem,stackengine}
\usepackage{newtxtext}
\usepackage[varvw]{newtxmath}
\allowdisplaybreaks
\usepackage{hyperref}
\hypersetup{
	colorlinks=true,
	linkcolor=blue,
	citecolor=red,%magenta,
}
\usepackage{appendix}

\mathtoolsset{showonlyrefs=true}
\theoremstyle{definition}
\newtheorem{definition}{Definition}[section]

\newtheorem{remark}{Remark}[section]

\newtheorem{problem}{Problem}[section]

\theoremstyle{plain}

\newtheorem{theorem}{Theorem}[section]

\newtheorem{corollary}{Corollary}[section]

\newtheorem{proposition}{Proposition}[section]

\newtheorem{lemma}{Lemma}[section]
\def\R{\mathbb{R}}

\renewcommand{\phi}{\varphi}
\newcommand\restr[2]{{% we make the whole thing an ordinary symbol
		\left.\kern-\nulldelimiterspace % automatically resize the bar with \right
		#1 % the function
		\vphantom{\big|} % pretend it's a little taller at normal size
		\right|_{#2} % this is the delimiter
}}

\let\phi\varphi
\let\epsilon\varepsilon

\usepackage{tikz-cd}
\tikzset{
	symbol/.style={
		draw=none,
		every to/.append style={
			edge node={node [sloped, allow upside down, auto=false]{$#1$}}}
	}
}

\def\R{\mathbb R}

\def\proscal3#1#2{<\!\!#1, #2\!\!>_{_{{\hskip-4pt\R^{3}}}}}

\def\vu{ u }
\def\vn{ \nabla }

\def\l2{L^2(\R^{n})}
\def\L2{L^2(\R^{2n})}
\def\supp{\operatorname{supp}}

\def\mat22#1#2#3#4{\begin{pmatrix}#1&#2\\ #3&#4\end{pmatrix}}

\def\mPl{\mathcal{P}^{\text{log}}} 
\def\Rt{\mathbb{R}^d}

\begin{document}

\title[Liouville Theorems Above the Critical ${9/2}$ Threshold]{Liouville Theorems Above the Critical $9/2$ Threshold for Stationary Navier-Stokes Equations}

	\author{Gast\'{o}n Vergara-Hermosilla}
\address{Institute for Theoretical Sciences, Westlake University, 310030 Hangzhou, Zhejiang, People's
Republic of China}
\email{gaston.v-h@outlook.com}	

	\maketitle

\begin{abstract}
We establish new Liouville-type theorems for the stationary Navier--Stokes equations in $\mathbb{R}^3$.
A central open problem in this context is whether the classical $L^{9/2}(\mathbb{R}^3)$ condition of G.~Galdi can be relaxed. In this note we show that this global integrability requirement can indeed be weakened.
More precisely, we prove that triviality already follows under assumptions of the form $u \in L^{9/2 + \varepsilon(\cdot)}(\mathbb{R}^3)$, where $\varepsilon(\cdot)>0$. 
As a consequence, we obtain a localized Liouville theorem: it is sufficient to impose this  integrability condition only at infinity, with no additional assumptions on the behavior of $u$ inside a compact set. This highlights that the mechanism enforcing triviality is purely asymptotic.
Our approach relies on a general uniqueness result in the framework of Lebesgue spaces with variable exponents, which naturally captures the coexistence of different integrability regimes across the domain.
\end{abstract}

%---------------------------------------------------
\section{Introduction and presentation of the results}
In this paper, we study the stationary Navier--Stokes equations in $\mathbb{R}^3$:
\begin{equation}\label{SNS}
-\Delta u + u \cdot \nabla u + \nabla P = 0, \quad 
\nabla \cdot u = 0,
\end{equation}
where $u$ denotes the velocity field and $P$ stands for the associated pressure. 
It is well known that solutions $(u,P)$ to \eqref{SNS} can be constructed in the spaces $\big(\dot{H}^1(\mathbb{R}^3), \dot{H}^{1/2}(\mathbb{R}^3)\big)$ (see, for instance, \cite[Theorem 16.2]{lemarie2016navier}). However, uniqueness in this class remains an open and dificuly problem. 
This motivates the following question, originally raised in \cite[Remark X.9.4]{galdi2011introduction} and \cite{Ser2016}.

\begin{problem}\label{Conjecture1}
Prove that any solution $u$ of \eqref{SNS} satisfying
\begin{equation}\label{Conjecture}
u \in \dot{H}^1(\mathbb{R}^3)
\qquad \text{and} \qquad
u(x) \to 0 \ \text{as } |x| \to +\infty,
\end{equation}
is identically equal to zero.
\end{problem}

In the following we briefly review some of the main progress on this problem.  By the Sobolev embedding theorem, any $u \in \dot{H}^1(\mathbb{R}^3)$ belongs to $L^6(\mathbb{R}^3)$, which already imposes a certain decay at infinity. However, this information alone does not seem sufficient to deduce  the triviality of the solution.
Over the years, several partial results have been obtained toward Problem \ref{Conjecture1}, showing that additional integrability or structural assumptions enforce $u \equiv 0$. One of the first results  about it is due to G.~Galdi \cite{galdi2011introduction}, who proved that $u \in L^{9/2}(\mathbb{R}^3)$ implies $u \equiv 0$. This condition was later relaxed by D.~Chae and J.~Wolf \cite{ChaeWolf}, who established that the weaker assumption
\[
\int_{\mathbb{R}^3} |u(x)|^{9/2} \bigl[\ln(2 + |u(x)|^{-1})\bigr]^{-1} dx < +\infty
\]
still guarantees $u \equiv 0$, providing a logarithmic improvement of Galdi’s result.
More recently, N.~Lerner in \cite{Lerner26} pointed that the global $L^{9/2}$ assumption can be relaxed by distinguishing low and high frequencies. Concretely, he proved that it is enough to require $u_{[0]} \in L^{9/2}(\R^3)$, where $u_{[0]}$ denotes the projection of $u$ onto the subspace of vector fields whose Fourier support contains a neighborhood of the origin.
  In a different direction, H.~Kozono, Y.~Terasawa, and Y.~Wakasugi proved in \cite{Kozonoetal}  that if the weak $L^{9/2}$ norm of $u$ satisfies
\[
\|u\|_{L^{9/2,\infty}} \le \delta \, (\nu \|\mathrm{curl}\, u\|_{L^2}^2)^{1/3}
\]
for sufficiently small $\delta$, then $u \equiv 0$. This result was later extended by G.~Seregin and W.~Wang in \cite{Sereginetwang}.
An alternative approach involves conditions on the Laplacian of $u$. Considering this approach, D.~Chae  showed  in \cite{chae14} that $\Delta u \in L^{6/5}(\mathbb{R}^3)$ already implies $u \equiv 0$. Now, some structural assumptions can also be used: in \cite{Seregin16}  G.~Seregin proved that if $u = \mathrm{curl}\, w$ with $w \in \mathrm{BMO}(\mathbb{R}^3)$, then $u$ must vanish. 

In this paper, we establish a new Liouville-type result that improves upon the classical $L^{9/2}$ condition. More precisely, we show that the triviality of solutions already follows from a weaker integrability assumption of the form $u \in L^{\frac{9}{2}+\varepsilon}(\mathbb{R}^3)$, with a spatially varying exponent. To the best of the author's knowledge, this provides the first result of this type. Our first main theorem is the following.

\begin{theorem}\label{thmpexplicit}
 Let $u \in \dot{H}^1(\mathbb{R}^3)$ be a   solution to \eqref{SNS} and $R_0>3/2$ fixed. Let $\varepsilon(\cdot)$ be a scalar function defined by  $\varepsilon(x)=\frac 32 $ for $|x|< R_0$ and $\varepsilon(x)=  \frac{3}{2} \frac{R_0}{|x|},$ for $|x| \geq  R_0.$ 
If in addition we assume  $u \in L^{\frac{9}{2}+\varepsilon(\cdot)}(\mathbb{R}^3)$, then $u \equiv 0$.
\end{theorem}

 \begin{remark}
As mentioned above, the problem would follow immediately if the
information provided by the Sobolev embedding could be upgraded to a
global $L^6(\mathbb{R}^3)$ condition strong enough to imply the
triviality of the solution. However, establishing such a result in the
whole space $\R^3$ remains an open problem. 
Theorem \ref{thmpexplicit} provides an alternative perspective; instead of requiring an uniform $L^6$ condition, we impose this level of integrability only on a fixed bounded region (since $9/2 + 3/2 = 6$ on $B(0,R_0)$), while allowing the exponent to decrease continuously toward the  value $9/2$ at infinity. In this way, the result interpolates between the Sobolev regime near the origin and the ``{\it critical}" regime at infinity.
\end{remark}

As a direct consequence of our first theorem, we obtain a Liouville-type result requiring $u \in L^{9/2 + \varepsilon(\cdot)}$ only outside an open set containing the origin. The result reads as follows:
\begin{corollary}\label{corollary-localized-assumption}
 Let $u \in \dot{H}^1(\mathbb{R}^3)$ be a   solution to \eqref{SNS} and $R_0>3/2$ fixed. Let $\overline \varepsilon(x)=  \frac  3 2 \frac{R_0}{|x|}$ be a scalar function defined on  $\{ |x| \geq  R_0 \}.$ 
If in addition, 
\[
\mathbf{1}_{\ \{ |x| \geq  R_0 \}  } \, u \in L^{9/2 + \overline \varepsilon(\cdot)}(\mathbb{R}^3),
\]
then $u \equiv 0$.
\end{corollary}

\noindent We now comment on this result and its relation to the existing literature.

\begin{remark}
Corollary \ref{corollary-localized-assumption} shows that the integrability condition can be completely localized at infinity. More precisely, no assumption is required on the behavior of $u$ inside the ball $B(0,R_0)$, and it is enough to impose a variable exponent condition only in the exterior region. 
Compared to the classical assumption $u \in L^{9/2}(\mathbb{R}^3)$, this result allows for a spatially varying exponent of the form $9/2 + \overline{\varepsilon}(x)$, where $\overline{\varepsilon}(x) = \frac{3}{2}\frac{R_0}{|x|}$ decays to zero as $|x| \to \infty$. In particular, the critical exponent $9/2$ is only required asymptotically, and can be approached from above at a quantified rate.
 We also emphasize that $R_0$ is arbitrary, so the condition may be imposed arbitrarily far from the origin.
 %This highlights that the mechanism behind the result is purely driven by the asymptotic behavior of the solution.
\end{remark} 

 \begin{remark}
This result can be compared with the recent work of N.~Lerner \cite{Lerner26}, where the $L^{9/2}$ condition is relaxed by separating low and high frequency contributions. In contrast, our approach is purely formulated in physical space and shows that it is sufficient to control the behavior of $u$ at infinity, without any explicit frequency decomposition.
\end{remark}

Theorem \ref{thmpexplicit} follows from a more flexible uniqueness result formulated in the setting of Lebesgue spaces with variable exponents. This framework naturally captures the coexistence of different integrability regimes, $L^6$ near the origin and $L^{9/2}$ at infinity, within a single functional setting. 
This result is the following.
\begin{theorem}\label{thm:over92}
Let $R_0 > 1$ be fixed, and let $u \in \dot{H}^1(\mathbb{R}^3)$ be a solution of \eqref{SNS}. Let $p:\mathbb{R}^3 \to \mathbb{R}$ be a variable exponent such that:
\begin{enumerate}
\item $p(x) = 6$ for all $x \in B(0,R_0)$,
\item $p$ is 
Lipschitz,  
 %continuous,
 radially decreasing, and satisfies $p(x) \ge \frac{9}{2}$ for all $x \in \mathbb{R}^3$,
\item there exists a constant $C \in [0, R_0^2)$ such that
$
|p(x) - \tfrac{9}{2}| \le \frac{C}{|x|}, \  \text{for all } |x| \ge R_0^2.
$
\end{enumerate}
If, in addition, $u \in L^{p(\cdot)}(\mathbb{R}^3)$, then $u \equiv 0$.
\end{theorem}

\noindent We now place this result in the context of related literature.
\begin{remark}
Theorem \ref{thm:over92} shows that the rigidity of the problem is not determined by a uniform integrability condition, but rather by how integrability is distributed across the domain. In particular, having $L^6$ control on arbitrarily large bounded regions, combined with a mild and continuous relaxation toward the critical exponent $9/2$ at infinity, is sufficient to ensure triviality. 
The framework of variable exponent spaces $L^{p(\cdot)}(\mathbb{R}^3)$ naturally captures this phenomenon, allowing for a smooth transition between different integrability regimes within a single functional setting.
\end{remark}

 \begin{remark}
The use of Lebesgue spaces with variable exponents has previously been
considered in the context of Liouville theorems for stationary
Navier--Stokes equations. In \cite{CV2023}, discontinuous exponents are
introduced on regions of infinite measure. However, such spaces do not
generally ensure the boundedness of the Riesz transforms, which are
required to recover the pressure from the velocity field. As a result,
conditions must be imposed simultaneously on both $u$ and $P$.
This difficulty is avoided in our setting. The continuous exponent
$p(\cdot)$ in Theorem \ref{thm:over92} belongs to a class for which the
Riesz transforms are bounded (see Definition \ref{globallogholder} and
Proposition \ref{prop:pinplog} below). This allows us to formulate the
condition solely in terms of the velocity field.
\end{remark} 

The rest of this paper is organized as follows. In Section \ref{sect:2}, we provide a brief review of variable exponent Lebesgue spaces, covering their definitions and key properties, along with the proof of several useful lemmas. Section \ref{sec:3} is devoted to the proof of our main results. \\

Given $R>1$, in what follows, and throughout this paper, we consider the following notation
\[
\mathcal{C}_R := \{ x \in \mathbb{R}^3 \ : \ R/2 < |x| < R \}.
\]

\section{Preliminaries}\label{sect:2}
To keep this paper reasonably self-contained, several results and definitions of variable Lebesgue spaces are recalled. To begin,  given a set $\Omega\subset\mathbb{R}^{n}$, let $\mathcal{P}(\Omega)$ be the set of all Lebesgue measurable functions $p(\cdot):\Omega\rightarrow[1,+\infty]$.  The elements of $\mathcal{P}(\Omega)$ are called variable exponent functions or simply variable exponents. For any $p(\cdot)\in \mathcal{P}(\Omega)$, we denote
$$
 p^{-}:=\operatorname{essinf}_{x\in\Omega}p(x), \ \  p^{+}:=\operatorname{esssup}_{x\in\Omega}p(x).
$$
Throughout this paper, we will assume  $1<p^-\leq p^+<+\infty$.\\

Given a domain $\Omega\subseteq\mathbb{R}^{n}$ and $p(\cdot)\in \mathcal{P}(\Omega)$, for a measurable function $u$, we consider 
\begin{equation}\label{equation2.2}
   \|u\|_{L^{p(\cdot)}}:=\inf\left\{\lambda>0:  \rho_{p(\cdot)}\left(\frac{u}{\lambda}\right)\leq 1\right\},
\end{equation}
where the modular function $\rho_{p(\cdot)}$ associated with $p(\cdot)$ is given by 
\begin{equation*}
   \rho_{p(\cdot)}(f):=
   \int_\Omega|u(x)|^{p(x)}dx.
\end{equation*}
If the set on the right-hand side of  \eqref{equation2.2} is empty then $\|u\|_{L^{p(\cdot)}} = +\infty$ by definition.
At this point, is interesting to note that,  if the exponent function $p(\cdot)$ is a constant, i.e. if $p(\cdot)=p\in [1,\infty)$, then we can obtain the usual norm  via the modular function $ \rho_{p}$.
\smallbreak

\begin{definition}\label{de2.2}
Given a domain $\Omega\subseteq\mathbb{R}^{n}$ and $p(\cdot)\in \mathcal{P}(\Omega)$, we define the variable exponent Lebesgue space $L^{p(\cdot)}(\Omega)$ as the set of   measurable functions $u$ such that   $\|u\|_{L^{p(\cdot)}}<+\infty$.
\end{definition}

\begin{remark}
Note that, $L^{p(\cdot)}(\Omega)$ is a Banach space associated with the norm $\|\cdot\|_{L^{p(\cdot)}}$.
\end{remark}

Next, we collect some properties of the variable exponent Lebesgue spaces. 
%--------------------------------
\begin{lemma}[H\"{o}lder inequality]
Consider $\Omega\subseteq \mathbb{R}^n$ and $p_1(\cdot),\,p_2(\cdot),\,p(\cdot)\in \mathcal{P}(\mathbb{R}^n)$ such that 
$\frac{1}{p(x)}=\frac{1}{p_1(x)}+\frac{1}{p_2(x)}$, for $x\in \Omega$. 
Then, given $u\in L^{p_1(\cdot)}(\Omega) $ and $v \in L^{p_2(\cdot)}(\Omega)$, the pointwise product $uv$ belongs to  $L^{p(\cdot)}(\Omega)$, and there exists  $C>0$ such that 
\begin{equation}\label{eq2.3}
\|uv\|_{L^{p(\cdot)}(\Omega)} \leq C\|u\|_{L^{p_1(\cdot)}(\Omega)}\|v\|_{L^{p_2(\cdot)}(\Omega)}.
\end{equation}
\end{lemma}
%--------------------------------
A proof of this result can be consulted in \cite[Corollary 2.28]{CRUZ} or  \cite[Lemma 3.2.20]{Diening_Libro}.
%%%%%%%%%%%%%%%%%%%%%%%%%%%%%%%%%%%%%%%%%%%%%%%%%%%
\begin{definition}\label{globallogholder}
Consider $\Omega\subseteq \mathbb{R}^d$ 
and a variable exponent $p(\cdot)\in \mathcal{P}(\Omega)$.
We say that  $ p(\cdot): \Omega \to \mathbb{R} $ is  locally log-Hölder continuous on $ \Omega $ if there exists $ C_1 > 0 $ such that
\begin{equation}\label{relation_loc_logholder}
|p(x) - p(y)| \leq \frac{C_1}{\log(e + 1/|x - y|)},
\end{equation}
for all $ x, y \in \Omega $. We say that $p(\cdot) $ satisfies the  log-Hölder decay condition if there exist $ p_\infty \in \mathbb{R} $ and a constant $ C_2 > 0 $ such that, for all $ x \in \Omega $ we have
\begin{equation}\label{relation_decay_logholder}
|p(x) -  p_\infty| \leq \frac{C_2}{\log(e + |x|)}
.
\end{equation}
 We say that $ p(\cdot) $ is  globally log-Hölder continuous in $ \Omega $ if it is locally log-Hölder continuous and satisfies the log-Hölder decay condition. 
Such class of exponents is denoted by 
$\mathcal{P}^{log}(\Omega)$. 
\end{definition}
%%%%%%%%%%%%%%%%%%%%%%%%%%%%%%%%%%%%%%%%%%%%%%%%%%%
At this point,  it is convenient to recall the following: 
given a measurable domain  $ \Omega \subset \Rt$ and a variable exponent $p(\cdot) \in \mathcal{P}(\Rt)$,  the notation 
$p_\Omega(\cdot)$ stands for the variable exponent restricted to the set $\Omega$, \emph{i.e.} ${p}_\Omega(\cdot)=  p  (\cdot)_{|_{\Omega}}$. 
%%%%%%%%%%%%%%%%%%%%%%%%%%%%%%%%%%%%%%%%%%%%%%%%%%%
\begin{lemma}\label{PropositionLpplusminus}
Consider a measurable set $\Omega\subset \R^3$ and $p(\cdot)\in \mathcal{P}(\Rt)$ a variable exponent, assume that we have $|\Omega|<+\infty$.  Then 
$$\|1\|_{L^{p_{\Omega}(\cdot) }(\Omega)}\leq 2\max\{|\Omega|^{\frac 1{ p^-}},|\Omega|^{\frac 1{p^+}}\}.$$
\end{lemma}
%%%%%%%%%%%%%%%%%%%%%%%%%%%%%%%%%%%%%%%%%%%%%%%%%%%
The proof of this result can be consulted in \cite[Lemma 3.2.12]{Diening_Libro}.  
%%%%%%%%%%%%%%%%%%%%%%%%%%%%%%%%%%%%%%%%%%%%%%%%%%%
\begin{lemma}\label{lemmaLinftyLpvariable}
Let $\Omega\subseteq \R^3$ and $p(\cdot)\in \mathcal{P}(\mathbb{R}^3)$ a variable exponent. Then, we have the space inclusion $L^{\infty} (\Omega) \subset L^{p_{\Omega}(\cdot)} (\Omega)$, if and only if $1\in  L^{p_{\Omega}(\cdot)} (\Omega) $ and  the following estimate follows
$$\|f\|_{L^{p_{\Omega}(\cdot)} (\Omega)}\leq \|f\|_{L^\infty(\Omega)}\|1\|_{L^{p_{\Omega}(\cdot)} (\Omega)}.$$
In particular, the embedding holds if $|\Omega|<+\infty$.
\end{lemma}
%%%%%%%%%%%%%%%%%%%%%%%%%%%%%%%%%%%%%%%%%%%%%%%%%%%
The proof of this result can be found in \cite[Proposition 2.43]{CRUZ}. 
%%%%
\begin{lemma}\label{lemma2.3}
For a bounded domain $\Omega\subset\mathbb{R}^{n}$ and two exponent functions $p(\cdot), q(\cdot)\in \mathcal{P}(\Omega)$ such that $1<p^{+}, q^{+}<+\infty$. Then  $L^{q(\cdot)}(\Omega)\hookrightarrow L^{p (\cdot)}(\Omega)\iff p (x)\leq q(x)$ almost everywhere. 
%Furthermore,  in this case we have
%\begin{equation}\label{eq2.4}
%   \|f\|_{L^{p_{1}(\cdot)}}\leq \left(1+|\Omega|\right)\|f\|_{L^{p_{2}(\cdot)}}.
%\end{equation}
\end{lemma}
%%%%
For a proof of this result, please see \cite[Corollary 2.48]{CRUZ}. 
In the next proposition we present nice relations between the norm of a function in $L^{p(\cdot)}$ and its modular function.
\begin{proposition} \label{proposition.utendzero}
 Given $\Omega \subseteq \R^n$ and $p(\cdot) \in \mathcal{P}(\Omega)$. If $\|u\|_{L^{p(\cdot)}(\Omega)} > 1$, then
\[
 \|u\|_{L^{p(\cdot)}(\Omega)} \leq \left(\int_{\Omega} |u(x)|^{p(x)}dx \right)^{1/p_-}.
\]
If $\|u\|_{L^{p(\cdot)}(\Omega)} \leq 1$, then
\[
 \|u\|_{L^{p(\cdot)}(\Omega)} \leq \left(\int_{\Omega} |u(x)|^{p(x)}dx \right)^{1/p_+}.
\] 
\end{proposition}
\noindent A proof of this result can be consulted in \cite[Chapter 2, page 25]{CRUZ}. We will use it for proving the following lemma.

\begin{lemma} \label{Lemma.utendzero}
Let $p(\cdot)\in \mathcal{P}(\R^3 )$  and $u\in L^{p(\cdot)}(\R^3 )$. 
Then, the following limit follows
\begin{equation}
\lim_{R\to +\infty} \| u\|_{L^{p(\cdot)}( \mathcal{C}_R  )} =0
.
\end{equation}
\end{lemma}
\begin{proof}
To begin note that, by considering that $ \mathcal{C}_R  \subset \{ \frac R 2 < |x| \}=:E_R$, we have 
\begin{equation}
\int_{ \mathcal{C}_R } |u(x)|^{p(x)}dx
\leq 
\int_{E_R} |u(x)|^{p(x)}dx<\infty.
\end{equation}
Since $|u(x)|^{p(x)} \in L^1(\mathbb{R}^3)$ and $E_R \downarrow \emptyset$, we have
$
\lim_{R \to \infty} \int_{E_R} |u(x)|^{p(x)}dx = 0,
$
by the dominated convergence theorem.
Hence, we conclude 
$$\int_{ \mathcal{C}_R } |u(x)|^{p(x)}dx \to  0\quad   
 \text{as } \ \  R \to \infty .$$
 
%Let $f(x):=|u(x)|^{p(x)} \in L^1(\mathbb{R}^3)$ and define 
%$f_R(x):=f(x)\chi_{E_R}(x)$. Then $f_R(x)\to 0$ pointwise and 
%$|f_R(x)|\le f(x)\in L^1(\mathbb{R}^3)$. Hence, by the dominated 
%convergence theorem,
%\[
%\lim_{R\to\infty}\int_{E_R}|u(x)|^{p(x)}dx=\lim_{R\to\infty}\int_{\mathbb{R}^3} f_R(x)\,dx=0.\]

%Thus, considering this fact and Proposition \ref{proposition.utendzero}, we conclude the desired limit.  
 % Thus, since
%\[\int_{ \mathcal{C}_R } |u(x)|^{p(x)}dx \to 0,\]
\noindent 
%Thus, there exists $R_0>0$ such that for all $R\ge R_0$,
In particular, there exists $R_0>0$ such that for all $R \ge R_0$,
$
\int_{ \mathcal{C}_R } |u(x)|^{p(x)}dx <1.
$
Therefore, by Proposition \ref{proposition.utendzero}, it follows that
\[
\|u\|_{L^{p(\cdot)}( \mathcal{C}_R )}
\le \left(\int_{ \mathcal{C}_R } |u(x)|^{p(x)}dx \right)^{1/p_+}
\quad \text{for all } R \ge R_0.
\]
Then,  by passing to the limit as $R\to\infty$, we conclude that
$
\|u\|_{L^{p(\cdot)}( \mathcal{C}_R )} \to 0.
$
 \end{proof}

% \newpage
 %%%%%%%%%
\section{Proof of the main results}\label{sec:3}
 %%%%%%%%%
 
We start by gathering some preliminary results that will be used throughout the proofs.
 
%%%%%%%%%%%%%
\begin{proposition}\label{prop:pinplog}
Consider the variable exponent $p(\cdot)$ in  Theorem~\ref{thm:over92}. 
Then $p(\cdot) \in \mathcal{P}^{\log}(\mathbb{R}^3)$.
\end{proposition}
%%%%%%%%%%%%%

\begin{proof}
In the following we prove that $p(\cdot)$ is globally log-Hölder continuous.

\medskip

\noindent\emph{Local log-Hölder continuity.}
Since $ {9}/{2}\leq p(x)\leq 6$ for all $x\in \R^3$, it follows   that $|p(x)-p(y)|$ is bounded above by $3/2$. This,  combined with the fact that $p$ is Lipschitz, directly implies the local log-Hölder continuity property.  
 
\noindent\emph{Log-Hölder decay.} To begin, we set $
p_\infty := 9/2.$ First note that, if  $ |x| \geq R_0^2 $, we know, by hypothesis that $|p(x) - 9/2| \leq C/{|x|}$.
Then, since the scalar function $f(t) = {\log(e + t)}/{t}$, is continuous on $[1,+\infty)$ and $ \lim_{t \to \infty} f(t) = 0 $, which makes it in particular bounded, we can write  
\[
  \frac{C}{|x|} \leq \frac{C_2}{\log(e + |x|)},
\]
for a suitable choice of the constant $C_2>0$. 
This fact implies that 
\[
|p(x) - 9/2|  \leq \frac{C_2}{\log(e + |x|)}.
\]
On the other hand, if $ |x| < R_0^2 $, using the boundedness $|p(x) - 9/2| \leq \frac{3}{2}$, and the fact that $\log(e + |x|) \leq \log(e + R_0^2)$, we can write 
\[
|p(x) - 9/2| \leq \frac{3}{2} \cdot \frac{\frac{3}{2} \log(e + R_0^2)}{\log(e + |x|)}.
\]
Thus, considering this information we prove the log-Hölder decay and we conclude the proof.
\end{proof}

At this point,  it is useful to recall that, $p_\Omega(\cdot)$ denotes the variable exponent restricted to a domain $\Omega \subset \Rt$, \emph{i.e.} $p_\Omega(\cdot)=  {p} (\cdot)_ { |_{\Omega} }$, and  
\[
p_{\Omega}^-={\mbox{inf ess}}_{x\in \Omega} \ p(x) \qquad \mbox{and}\qquad p_{\Omega}^+
=
 {\text{sup ess}}_{x\in \Omega} \ p(x)
.
\]
\begin{lemma}\label{lem:PCsup_asymptotics}
Let $p(\cdot)$ be the variable exponent defined in Theorem \ref{thm:over92} and $R\geq 2R_0^2$. 
Then, 
there exists a sequence $(\varepsilon_R)$ and $C>0$ such that
\[
p_{\mathcal{C}_R}^+ = \frac{9}{2} + \varepsilon_R, 
\quad 0<\varepsilon_R \le \frac{2C}{R}<1.
\]
In particular,
$\label{O1R}
p_{\mathcal{C}_R}^+ =  {9}/ {2} + O\!\left( {1} / {R}\right).
$
\end{lemma}
\begin{proof}
Let $x\in \mathcal{C}_R$. 
Since $p(\cdot)$ is continuous and radially decreasing,  
there exists a decreasing and continuous function
$\tilde{p} : [0, \infty) \to [9/2, 6]$, 
such that $ p(x) = \tilde{p}(|x|) $ for all $ x \in \mathbb{R}^3 $.
Then, for each $ x \in \mathcal{C}_R $ we have $|x| > R/2$, hence, 
\[
p(x) = \tilde{p}(|x|) \leq \tilde{p}(R/2) %= p(R/2).
,
\]
and we conclude 
\[
\text{ess sup}_{x \in \mathcal{C}_R} p(x) \leq  \tilde{p}(R/2) 
.
\] 
Let $ \epsilon > 0 $. By continuity of $ \tilde{p} $ at $ R/2 $, there exists $ \delta > 0 $ such that, 
$
r \in (R/2, R/2 + \delta) $ implies $|\tilde{p}(r) - \tilde{p}(R/2)| < \epsilon
,
$
in particular,  if
$
r \in (R/2, R/2 + \delta) $, then  $\tilde{p}(r) >\tilde{p}(R/2) - \varepsilon.$
To continue, we define the set
\[
A_\delta := \{x \in \mathbb{R}^3 : R/2 < |x| < R/2 + \delta\}.
\]
At this point we must stress the fact that, $ A_\delta \subset \mathcal{C}_R $ and, for all $ x \in A_\delta $, and we can write
\[
p(x) > \tilde{p}(R/2) - \epsilon.
\]
Moreover, $ A_\delta $ has positive Lebesgue measure, since
\[
| A_\delta | 
=
 \frac{4\pi}{3} ((R/2 + \delta)^3 - (R/2)^3) > 0.
\]
As a consequence of this, we get that the set $
\{x \in \mathcal{C}_R : p(x) > \tilde{p}(R/2) - \varepsilon\}$ has positive measure. Then, considering the definition of essential supremum, we can write $\operatorname{ess\,sup}_{x \in \mathcal{C}_R} p(x) \geq \tilde{p}(R/2) - \epsilon$. Thus, since  $ \varepsilon > 0 $ is arbitrary, we conclude
\[
\operatorname{ess\,sup}_{x \in \mathcal{C}_R} p(x) \geq \tilde{p}(R/2),
\]
and then, 
\[
p_{\mathcal{C}_R}^+ = \text{ess sup}_{x \in \mathcal{C}_R} p(x) =\tilde{p}(R/2).
\]
Now, to conclude  the asymptotic estimate, note that  
since $ R \geq 2R_0^2 $, we have $ R/2 \geq R_0^2 $, and by hypothesis, we know that there exist $C<R^2_0\leq R/2$ such that 
\[
\left| \tilde p(R/2) - \frac{9}{2} \right| \leq \frac{C}{R/2} = \frac{2C}{R}<1.
\]
Thus, defining %, for instance,  
  $\epsilon_R := \tilde p(R/2) - \frac{9}{2}$ 
%we obtain
%\[
%p_{\mathcal{C}_R}^+ = \frac{9}{2} + \varepsilon_R, \quad |\varepsilon_R| \leq \frac{2C}{R}<1,
%\]
%and
 we deduce the desired results. 
\end{proof}

\begin{lemma}\label{lem:exponent_control}
Under the assumptions of Lemma~\ref{lem:PCsup_asymptotics}, define $\alpha_R := 2 - \frac{9}{P_{\mathcal{C}_R}^+}.$ 
Then
\[
\alpha_R = O\!\left(\frac{1}{R}\right),
\]
and moreover
\[
R^{\alpha_R} = 1 + O\!\left(\frac{\ln R}{R}\right).
\]
\end{lemma}

\begin{proof}
By Lemma~\ref{lem:PCsup_asymptotics}, we know that 
\[
p_{\mathcal{C}_R}^+ = \frac{9}{2} + \varepsilon_R, \quad \varepsilon_R = O\!\left(\frac{1}{R}\right), \quad \varepsilon_R  <1.
\]
Provided with this, we can write  
\[
\frac{9}{p_{\mathcal{C}_R}^+}
= \frac{9}{\frac{9}{2} + \varepsilon_R}
= \frac{2}{1 + \frac{2\varepsilon_R}{9}}.
\]
Using Taylor expansion for $t= {2\varepsilon_R}/ {9}<1$, we can write
\[
\frac{1}{1+t} = 1 - t + O(t^2),
\]
and we obtain
\[
\frac{9}{p_{\mathcal{C}_R}^+}
= 2 \left( 1 - \frac{2\varepsilon_R}{9} + O(\varepsilon_R^2) \right)
= 2 - \frac{4\varepsilon_R}{9} + O(\varepsilon_R^2).
\]
Hence, we conclude 
\[
\alpha_R 
=
 2 - \frac{9}{p_{\mathcal{C}_R}^+}
= \frac{4\varepsilon_R}{9} + O(\varepsilon_R^2)
= O\!\left(\frac{1}{R}\right).
\]
Now, considering the identity 
\[
R^{\alpha_R} = \exp(\alpha_R \ln R),
\]
and the fact that $\alpha_R = O(1/R)$, we obtain  
\[
\alpha_R \ln R = O\!\left(\frac{\ln R}{R}\right) \to 0
\quad  \text{ as } R \to \infty.
\]
Thus, by  considering the expansion $e^s -1=  O(s)$ as $s \to 0$ (with $s=\alpha_R \ln R$),  we can write 
\[
R^{\alpha_R} = 1 + O\!\left(\frac{\ln R}{R}\right).
\]
With this we finish the proof. 
\end{proof}

\begin{proposition}\label{prop:main_limit_thmover}
Let $p(\cdot)$ satisfy the assumptions of Lemma~\ref{lem:PCsup_asymptotics}, and let
$f:[1,\infty) \to \mathbb{R}$ be such that $
f(R) \to 0 \,  \text{ as } R \to \infty.$ 
Then
\[
\lim_{R\to\infty} R^{2 - \frac{9}{P_{\mathcal{C}_R}^+}} f(R) = 0.
\]
\end{proposition}

\begin{proof}
By Lemma~\ref{lem:exponent_control}, we can write
\[
R^{2 - \frac{9}{P_{\mathcal{C}_R}^+}} = R^{\alpha_R}
= 1 + \delta_R,
\]
where $\delta_R = O\!\left(\frac{\ln R}{R}\right)$.
% and, in particular, $\delta_R \to 0$ as $R \to \infty$. Hence, $(\delta_R)$ is bounded for $R$ large enough.
Therefore,  we get 
\[
R^{2 - \frac{9}{P_{\mathcal{C}_R}^+}} f(R)
= (1 + \delta_R) f(R)
= f(R) + \delta_R f(R).
\]
Now, since $f(R) \to 0$ and $(\delta_R)$ is bounded for $R\geq 1$, we have
\[
|\delta_R f(R)| \le C |f(R)| \to 0 \quad \text{as } R \to \infty.
\]
Thus, we obtain
\[
R^{2 - \frac{9}{P_{\mathcal{C}_R}^+}} f(R) \to 0 \quad \text{as } R \to \infty
.
\]
\end{proof}

\begin{proof}[Proof of Theorem \ref{thm:over92}]
 
   Let  $\vu\in \dot H^1(\R^3) %L^2_{loc}(\R^3 )
   $ be a solution of the 3d stationary Navier-Stokes equations. By considering  Theorem 2.51 in \cite{CRUZ}, we know the inclusions
$$L^{p(\cdot)}(\R^3 )\subset L^{p^-}(\R^3 )+L^{p^+}(\R^3 )
\subset L^{p^-}_{loc}(\R^3 )+L^{p^+}_{loc}(\R^3 )
.$$
Now, by the hypothesis assumed on the variable exponent $p(\cdot)$ we  have $3<9/2\leq p^-\leq p^+\leq 6 $, and  then,  we can deduce  
$$u\in L^{p(\cdot)}(\R^3 ) \subset L^3_{loc}(\R^3 ).$$
Thus, by Theorem X.1.1 in \cite{galdi2011introduction} we conclude that $(\vu,P)$ is in fact a couple of regular functions. 
Now, let $\phi\in {\mathcal{C}}^\infty_0 (\mathbb{R}^3) $ be a smooth function such that $0\leq \phi \leq 1$, $\phi(x)=1$ if $|x|<\frac{1}{2}$, $\phi(x)=0$ if $|x|>1$. Given $R>1$, we consider the function $\phi_R(x)=\phi( {x}/ {R})$. Thus, $\phi_R(x)=1$ on $\{|x|<   R / 2 \}$ and $\phi_R(x)=0$ on $\{|x|\geq R \}$.
By testing the stationary Navier-Stokes equations \eqref{SNS} with $\phi_R \vu$ and since  $\supp (\phi_R \vu)\subset B_R=B(0,R)$,  we get 
\begin{equation}\label{eq. 3_92_exterior_bola}
\int_{B_R}-\Delta \vu \cdot\left(\phi_R \vu\right)+(\vu \cdot \vn) \vu \cdot\left(\phi_R \vu\right)+\vn P \cdot\left(\phi_R \vu\right) d x=0.
\end{equation}
Note that, since the couple $(\vu,P)$ is regular, the terms involved in the equality above are well-defined. Then, by using the divergence-free condition $\nabla \cdot \vu=0$ and integration by parts, we obtain 
\begin{multline}
0=
\int_{B_R}-\Delta \vu \cdot\left(\phi_R \vu\right) dx
+
\int_{B_R}(\vu \cdot \vn) \vu \cdot\left(\phi_R \vu\right) dx
+
\int_{B_R} \vn P \cdot\left(\phi_R \vu\right) d x 
\\
= 
-\int_{B_R} \Delta \phi_R\left(\frac{|\vu|^2}{2}\right) d x+\int_{B_R} \phi_R|\vn \otimes \vu|^2 dx
-\int_{B_R}\vn\phi_R \cdot\left(\frac{|\vu|^2}{2} \vu\right) d x
-\int_{B_R} \vn \phi_R \cdot(P \vu)dx
.
\end{multline}
Thus, we get the identity 
\begin{equation*}
\int_{B_R} \phi_R|\vn \otimes \vu|^2 dx=\int_{B_R} \Delta \phi_R \frac{|\vu|^2}{2} dx+
\int_{B_R} \vn \phi_R \cdot \left(P+\frac{|\vu|^2}{2}\right) \vu dx.
\end{equation*}
Then, 
considering that $\phi_R(x)= 1$ if $|x|< \frac R 2$, we get 
\begin{equation}\label{ineqBase92exteriorbola}
\int_{B_{\frac{R}{2}}}|\vn \otimes \vu|^2 dx \leq \int_{B_R} \Delta \phi_R \frac{|\vu|^2}{2} dx+ \int_{B_R} \vn \phi_R \cdot \left(P+\frac{|\vu|^2}{2}\right) \vu dx 
=: 
I_1(R)+I_2(R)
. 
\end{equation}
In  the following we will prove 
%\begin{equation}\label{LimitsAlphaBeta_92_exterior_bola}\displaystyle
$\lim_{R\to +\infty}  |I_1(R)|=\lim_{R\to +\infty}  |I_2(R)|=0.$
%\end{equation}
%In fact: 
%%%%%%%%%%%%%%%%%%%%%%

\noindent {\bf 1) Limit for $I_1(R)$}. For studying the term $I_1(R)$ in \eqref{ineqBase92exteriorbola}, the H\"older inequality with 
\begin{equation}\label{variable_exponents_1_92_exterior_bola}
1=\frac{2}{p(\cdot)}+\frac{1}{q(\cdot)}
\end{equation} yields the following estimate
\footnote{Considering the definition of the cut-off function  $\phi_R$   and Lemma  \ref{lemmaLinftyLpvariable} is straightforward to see that such functions and its partial derivatives belongs to the variable Lebesgue spaces considered here.} 
\begin{equation}\label{18aveq32_T1_92_exterior_bola}
|I_1(R)|\leq 
C\|\Delta \phi_R \|_{L^{q(\cdot)}(  \mathcal{C}_R )}\|\vu\|_{L^{p(\cdot)}(  \mathcal{C}_R )}^2.
\end{equation}
 In order to control the quantity $\|\Delta \phi_R \|_{L^{q_{ \mathcal{C}_R }(\cdot)}( \mathcal{C}_R )}$ above,
  by  Lemma \ref{lemmaLinftyLpvariable} we can write
\begin{equation}\label{EstimationNormeLinfiniTheta_R1_92_exterior_bola}
\|\Delta \phi_R \|_{L^{q_{ \mathcal{C}_R }(\cdot)}( \mathcal{C}_R )}\leq C\|\Delta \phi_R  \|_{L^{\infty}  ( \mathcal{C}_R )}\|1\|_{L^{q_{ \mathcal{C}_R }(\cdot)}( \mathcal{C}_R )}.
\end{equation}
Now, considering  the definition of  $\phi_R$,  we get 
$$
\|\Delta \phi_R  \|_{L^{\infty}
(  \mathcal{C}_R  )
}\leq CR^{-2}$$
 and we obtain
$$\|\Delta \phi_R \|_{L^{q_{
 \mathcal{C}_R  }
(\cdot)}( \mathcal{C}_R  )}\leq C R^{-2}\|1\|_{L^{q_{ \mathcal{C}_R  }(\cdot)}( \mathcal{C}_R  )}.$$
Then, by considering Lemma \ref{PropositionLpplusminus} to dealing with $\|1\|_{L^{q_{ \mathcal{C}_R }(\cdot)}( \mathcal{C}_R )}$, we can write
\begin{align}\label{EstimationC01_92_exterior_bola}
\|\Delta \phi_R \|_{L^{q_{ \mathcal{C}_R }(\cdot)}( \mathcal{C}_R )}&\leq
 C R^{-2}\max \{| \mathcal{C}_R |^{\frac 3 {q_{ \mathcal{C}_R }^- }}, | \mathcal{C}_R |^{\frac 3 {q_{ \mathcal{C}_R }^+}}\}
 .
\end{align}
Now, by stressing the fact   that 
$|
 \mathcal{C}_R 
|=CR^3$ and $R>1$, we obtain
\begin{align}\label{EstimationC1_92_exterior_bola}
\|\Delta \phi_R \|_{L^{q_{ \mathcal{C}_R }(\cdot)}( \mathcal{C}_R )}&\leq
 C\max \{R^{-2+\frac 3{q_{ \mathcal{C}_R }^-}}, R^{-2+\frac 3 {q_{ \mathcal{C}_R }^+}}\}
=
  C R^{-2+\frac 3{q_{ \mathcal{C}_R }^-}} 
  .
\end{align}
and thus, we get  
\begin{equation}\label{estimateI1_92_exterior_bola}
|I_1(R)|\leq 
  C R^{-2+\frac 3{q_{ \mathcal{C}_R }^-}} 
\|\vu\|_{L^{p(\cdot)}(  \mathcal{C}_R )}^2.
\end{equation}
Then, by mean of \eqref{variable_exponents_1_92_exterior_bola} we can recast the previous expression as 
\begin{equation}\label{estimateI122}
|I_1(R)|\leq 
  C R^{1-\frac 6{p_{ \mathcal{C}_R }^+}} 
\|\vu\|_{L^{p(\cdot)}(  \mathcal{C}_R  )}^2.
\end{equation}
Now, stressing the fact that $p_{ \mathcal{C}_R }^+ \leq p^+=6 $, $f(s)=R^s$ is an increasing function if $R>1$, and $\|\vu\|_{L^{p(\cdot)}(  \mathcal{C}_R  )}\to 0$ as  $R\to +\infty$ (see Lemma \ref{Lemma.utendzero}), we conclude 
$
I_1(R) \to_{R\to +\infty} 0. 
$

\noindent {\bf 2) Limit for $I_2(R)$}. Note that, by mean of  the definition of $\phi_R$ we know that $\supp( \vn \phi_R)\subset \mathcal{C}_R $. Thus, we can write 
\begin{eqnarray}
\left|I_2(R)\right|&=&\left|\int_{B_R} \vn \phi_R \cdot \left(P+\frac{|\vu|^2}{2}\right) \vu dx\right|\notag\\
& \leq &\frac{1}{2}  
\int_{ \mathcal{C}_R }|\vn \phi_R||\vu|^3dx + \int_{ \mathcal{C}_R }|\vn \phi_R||P \| \vu| dx 
=: I_{21}(R)+I_{22}(R)
.\label{aug_31T1_92_exterior_bola}
\end{eqnarray}
With this at hand, in the following our objective is to prove 
$$
\displaystyle \lim_{R\to +\infty}I_{21}(R)= \lim_{R\to +\infty} I_{22}(R)=0.
$$
To deal with the term $I_{21}(R)$, by the H\"older inequality with 
\begin{equation}\label{variable_exponents_2_92_exterior_bola}
1=\frac{3}{p(\cdot)}+\frac{1}{r(\cdot)}
\end{equation}
we write 
\begin{align}
I_{21}(R)=\int_{ \mathcal{C}_R }|\vn \phi_R||\vu|^3dx 
&\leq C
\|\vn\phi_R \|_{L^{ r(\cdot)}( \mathcal{C}_R )}\| |\vu|^3 \|_{L^{\frac{p(\cdot)}{3}}( \mathcal{C}_R )},
\\
&\leq C\|\vn\phi_R \|_{L^{ r(\cdot)}( \mathcal{C}_R )}\|\vu \|_{L^{p(\cdot)}( \mathcal{C}_R )}^3.\label{2mayeq3T1_92_exterior_bola}
\end{align}
Since $ \|\vn \phi_R  \|_{L^{\infty}}\leq CR^{-1}$ and $R>1$, following the same ideas than before, we obtain  
\begin{align}
\| \vn \phi_R \|_{L^{r(\cdot)}( \mathcal{C}_R )}
&\leq
 C\max \{R^{-1+\frac 3{r_{ \mathcal{C}_R }^-}}, R^{-1+\frac 3 {r_{ \mathcal{C}_R }^+}}\} 
=
 C R^{-1+\frac 3{r_{ \mathcal{C}_R }^-}}
 ,
  \label{EstimationBeta1C_92_exterior_bola}
\end{align}
and then we get 
\begin{equation}\label{estimateI21_92_exterior_bola}
I_{21}(R)\leq 
  C R^{-1+\frac 3{r_{ \mathcal{C}_R }^-}}
\|\vu \|_{L^{p(\cdot)}(\R^3 )}^3. 
\end{equation}
Thus, by mean of \eqref{variable_exponents_2_92_exterior_bola}, we can write 
\begin{equation}\label{estimateI2122_92_exterior_bola}
I_{21}(R)\leq 
  C R^{2-\frac 9{p_{ \mathcal{C}_R }^+}}
\|\vu \|_{L^{p(\cdot)}(  \mathcal{C}_R  )}^3. 
\end{equation}
\begin{remark}\label{remark_92_exterior_bola_2}
Then, since  
$\|\vu\|_{L^{p(\cdot)}(  \mathcal{C}_R  )}\to 0$ as  $R\to +\infty$, by considering Proposition \ref{prop:main_limit_thmover} we obtain
$$\lim_{R\to +\infty} R^{2-\frac 9{p_{ \mathcal{C}_R }^+}}
\|\vu \|_{L^{p(\cdot)}(  \mathcal{C}_R  )}^3 = 0.$$
\end{remark}
\noindent Thus, using this remark we conclude  $
I_{21}(R) \to_{R\to +\infty} 0. $

%%%%%
Now we analyze the term $I_{22}(R)$. Considering H\"older inequalities with $\frac{1}{p(\cdot)}+\frac{2}{p(\cdot)}+\frac{1}{r(\cdot)}=1$ and arguing for the cut-off function in the same manner than before, we  get the estimates 
\begin{align}
I_{22}(R)=\int_{ \mathcal{C}_R }|\vn \phi_R||P \| \vu| dx
&\leq C \| \vn\phi_R \|_{L^{r(\cdot)}( \mathcal{C}_R )} \|P \|_{L^{\frac{p(\cdot)}{2}}( \mathcal{C}_R )}\|\vu \|_{L^{p(\cdot)}( \mathcal{C}_R )} 
\\
&\leq 
 C R^{-1+\frac 3{r_{ \mathcal{C}_R }^-}}
\|P \|_{L^{\frac{p(\cdot)}{2}}( \mathcal{C}_R )}\|\vu \|_{L^{p(\cdot)}(  \mathcal{C}_R )}
\\
&=
C R^{2-\frac 9{p_{ \mathcal{C}_R }^+}}
\|P \|_{L^{\frac{p(\cdot)}{2}}( \R^3)}
\|\vu \|_{L^{p(\cdot)}(  \mathcal{C}_R )}\label{estimateI22_92_exterior_bola}
.
\end{align}
Now, in order to get 
the limit for $I_{22}$, we stress the fact that,  $p(\cdot)\in\mPl(\R^3) $ by Proposition \ref{prop:pinplog}  
 and, by using the divergence-free property of $\vu$, we have the classical relationship for $P$: 
$$\displaystyle{P=\sum_{i,j=1}^3\mathcal{R}_i\mathcal{R}_i(u_iu_j)}
,
$$ 
where the $\mathcal{R}_i$ stands for the   Riesz transforms.
Then,
gathering this relationship with the hypothesis $\vu\in L^{p(\cdot)}(\R^3 )$ and the fact that 
 the Riesz transform are bounded and continuous in $L^{p(\cdot)}$ spaces provided that $p(\cdot)\in\mPl(\R^3) $ (see for instance \cite[Section 12.4]{Diening_Libro}),  we conclude 
 \begin{equation}
 \|P \|_{L^{\frac{p(\cdot)}{2}}( \R^3)}
 \leq C 
 \|\vu \|^2_{L^{p(\cdot)}( \R^3 )}
 .
 \end{equation} 
 Considering this last inequality into \eqref{estimateI22_92_exterior_bola}, we obtain 
 \begin{align}
I_{22}(R)\leq 
C R^{2-\frac 9{p_{ \mathcal{C}_R }^+}}
 \|\vu \|^2_{L^{p(\cdot)}( \R^3 )}
\|\vu \|_{L^{p(\cdot)}(  \mathcal{C}_R )}
.
\end{align}
Thus, considering Remark \ref{remark_92_exterior_bola_2},  we conclude 
$
I_{22}(R) \to_{R\to +\infty} 0. 
$
%\end{itemize}%of the \begin{itemize} for the Limits

With this information at hand, we can conclude from the estimate (\ref{ineqBase92exteriorbola}): 
\begin{equation}\label{LimiteSobolev_92_exterior_bola}
\lim_{R\to +\infty}\int_{B_{\frac{R}{2}}}|\vn \otimes \vu|^2 dx=\|\vu\|_{\dot{H}^1(\R^3)}=0,
\end{equation}
from which we deduce, by considering Sobolev embeddings, that $\|\vu\|_{L^6}=0$ and thus  $\vu= 0$. 
\end{proof}

%\newpage

\begin{proof}[Proof of Theorem \ref{thmpexplicit}]

To begin, we define the variable exponent $p(\cdot)= {9} / {2} + \varepsilon(\cdot)$. 
Note  that we have $p(x)> {9} / {2}$ or all $x\in \R^3$,  and that $p$ is radial. Furthermore, it is straightforward to see that $p$ is Lipschitz (with Lipschitz constant $C_L = 3 / (2R_0) $).
 Now,  for $ |x| > R_0 $, the function $ \frac{R_0}{|x|} $ is  decreasing, hence $ p(\cdot) $ is  decreasing. On the other hand, we stress the fact that, for $ |x| \geq R_0^2 $, we can write 
\[
\left| p(x) - \frac{9}{2} \right| =
 \left| \frac{3}{2} \frac{R_0}{|x|} \right| = C |x|^{-1},
\]
where $ C = \frac{3}{2} R_0 <R_0^2$ (since $R_0>3/2$ by hypothesis).

With this information at hand, we note that this variable exponent $p(\cdot)$ fulfill the hypothesis of Theorem \ref{thm:over92}, then if the solution $u$ belongs to $L^{p(\cdot)} (\R^3)$, we conclude that $u=0$.
\end{proof}

%\newpage

%In the following Corollary we present a natural and nice consequence of  Theorem \ref{thm:over92}. 
 
 \begin{proof}[Proof of Corollary \ref{corollary-localized-assumption}]
In the folllowing, we will show that $
u \in L^{\frac{9}{2}+\varepsilon(\cdot)}(\mathbb{R}^3)$, 
where \( \varepsilon(\cdot) \) is as in Theorem~\ref{thmpexplicit}.
To this end, we decompose the domain as
\[
\mathbb{R}^3 = B(0,R_0) \cup \{ |x| \geq R_0 \}.
\]
Note that, since \( u \in \dot{H}^1(\mathbb{R}^3) \), the Sobolev embedding yields \( u \in L^6(\mathbb{R}^3) \), and in particular \( u \in L^6(B(0,R_0)) \).  
By definition, \( \varepsilon(x) = \frac{3}{2} \) for \( |x| < R_0 \), hence \( \frac{9}{2} + \varepsilon(x) = 6 \) in this region. 
Thus, it follows that $u \in L^{\frac{9}{2}+\varepsilon(\cdot)}(B(0,R_0)).$
On the other hand, by assumption of the corollary we have
$ u \in L^{\frac{9}{2}+\overline{\varepsilon}(\cdot)}(\{ |x| \geq R_0 \}),$
and by construction \( \overline{\varepsilon}(x) = \varepsilon(x) \) for \( |x| \geq R_0 \). Therefore, we can write $
u \in L^{\frac{9}{2}+\varepsilon(\cdot)}(\{ |x| \geq R_0 \}).
$
Then, by combining the two regions, we conclude 
$
u \in L^{\frac{9}{2}+\varepsilon(\cdot)}(\mathbb{R}^3).
$
Thus, as all the hypotheses of Theorem~\ref{thmpexplicit} are fulfilled, we apply it and we conclude \( u \equiv 0 \).
\end{proof}
 
\paragraph{\bf Acknowledgements} The author warmly thanks  Pierre-Gilles Lemarié-Rieusset, Hedong Hou and Alexey Cheskidov  for their valuable advice and insightful comments.
 
\paragraph{\bf Datasets} Data sharing does not apply to this article as no datasets were generated or analyzed during the
current study.
 
\paragraph{\bf  Conflict of interest} In addition, the author declares no conflict of interest and confirms being the sole contributor to this paper.
 
%------------------
%------ end of document 
%------------------
%\newpage


\begin{thebibliography}{99}
%\addcontentsline{toc}{section}{References}
\bibitem{chae14}
{\sc D.~Chae}, {\em Liouville-type theorems for the forced {E}uler equations
  and the {N}avier-{S}tokes equations}, Comm. Math. Phys., 326 (2014),
  pp.~37--48.
%
\bibitem{ChaeWolf}
{\sc D.~Chae and J.~Wolf}, {\em On {L}iouville type theorems for the steady
  {N}avier-{S}tokes equations in {$\mathbb{R}^3$}}, J. Differential Equations,
  261 (2016), pp.~5541--5560.
%
\bibitem{CV2023} D. Chamorro,  G. Vergara-Hermosilla, Liouville type theorems for stationary Navier-Stokes equations with Lebesgue spaces of variable exponent, Documenta Mathematica (2025).
%
\bibitem{CRUZ}
{\sc D.~V. Cruz-Uribe and A.~Fiorenza}, {\em Variable {L}ebesgue spaces:
Foundations and harmonic analysis}, Springer Science \& Business Media, 2013.
%
\bibitem{Diening_Libro}
{\sc L.~Diening, P.~Harjulehto, P.~H{\"a}st{\"o} and M.~Ruzicka}, {\em
Lebesgue and Sobolev spaces with variable exponents}, Springer, 2011.
%
\bibitem{galdi2011introduction}
{\sc G.~Galdi}, {\em An introduction to the mathematical theory of the
{N}avier-{S}tokes equations: Steady-state problems}, Springer Science \&
Business Media, 2011.
%
\bibitem{Kozonoetal}
{\sc H.~Kozono, Y.~Terasawa, and Y.~Wakasugi}, {\em A remark on
  {L}iouville-type theorems for the stationary {N}avier-{S}tokes equations in
  three space dimensions}, J. Funct. Anal., 272 (2017), pp.~804--818.
%
\bibitem{lemarie2016navier}
{\sc P.~G. Lemari{\'e}-Rieusset},
{\em The {N}avier-{S}tokes problem in the 21st century}, CRC press, 2016.
%
%
\bibitem{Lerner26}
{\sc N.~
Lerner},   {\em Wiener Algebras Methods for Liouville Theorems on the Stationary Navier-Stokes System}, arXiv preprint arXiv:2601.13916  (2026).
%
\bibitem{Ser2016}  
G. \textsc{Seregin}, \emph{A Liouville type theorem for steady-state Navier-Stokes equations}.  
J. {\'E}.D.P., Expos{\'e} no IX, (2016).
%
\bibitem{Seregin16}
{\sc G.~Seregin}, {\em Liouville type theorem for stationary {N}avier-{S}tokes
  equations}, Nonlinearity, 29 (2016), pp.~2191--2195.
%
\bibitem{Sereginetwang}
{\sc G.~Seregin and W.~Wang}, {\em Sufficient conditions on {L}iouville type
  theorems for the 3{D} steady {N}avier-{S}tokes equations}, Algebra i Analiz,
  31 (2019), pp.~269--278.
%
\end{thebibliography}
 \end{document}